\documentclass{amsart}
\usepackage[utf8]{inputenc}
\usepackage[english]{babel}

\usepackage{xcomment}

\usepackage{amsmath,amssymb,amsthm,amsrefs}
\usepackage{mathtools}

\usepackage{caption,subcaption} 
\usepackage[all,curve]{xy} 

\usepackage{hyperref}
\usepackage{cleveref}
\usepackage{enumitem}

\usepackage{tikz}

\usetikzlibrary{shapes,arrows}
\usetikzlibrary{decorations.markings}
\usetikzlibrary{arrows.meta}
\usetikzlibrary{patterns}

\newtheorem{thm}{Theorem}[section]
\newtheorem*{thm*}{Theorem}
\newtheorem{cor}[thm]{Corollary}
\newtheorem*{cor*}{Corollary}
\newtheorem{prop}[thm]{Proposition}
\newtheorem*{prop*}{Proposition}
\newtheorem{lemma}[thm]{Lemma}
\newtheorem*{lemma*}{Lemma}
\theoremstyle{definition}
\newtheorem{defn}[thm]{Definition}
\newtheorem*{defn*}{Definition}
\newtheorem{stmt}[thm]{Statement}
\newtheorem{openQuestion}{Open Question}[section]
\theoremstyle{remark}
\newtheorem{rem}[thm]{Remark}
\newtheorem{obs}[thm]{Observation}
\numberwithin{equation}{section}

\DeclarePairedDelimiter{\bracket}{[}{]}
\DeclarePairedDelimiter{\braces}{\{}{\}}

\newcommand{\set}[1]{\braces*{#1}}
\DeclarePairedDelimiterX{\setbuild}[2]{\{}{\}}{{#1}\;\delimsize\vert\;{#2}}
\DeclarePairedDelimiterX{\setbuildc}[2]{\{}{\}}{{#1}\vcentcolon{#2}}

\renewcommand{\models}{\vDash}
\newcommand{\nmodels}{\nvDash}
\DeclareSymbolFont{AMSb}{U}{msb}{m}{n}

\DeclareMathSymbol{\N}{\mathbin}{AMSb}{"4E} 
\DeclareMathSymbol{\Z}{\mathbin}{AMSb}{"5A} 
\DeclareMathSymbol{\R}{\mathbin}{AMSb}{"52} 
\DeclareMathSymbol{\Q}{\mathbin}{AMSb}{"51} 
\DeclareMathSymbol{\C}{\mathbin}{AMSb}{"43} 
\makeatletter
\newcommand{\nfs}{\@ifnextchar.{}{.\@}}
\makeatother

\newcommand{\leT}{\le_{\mathrm{T}}}
\newcommand{\geT}{\ge_{\mathrm{T}}}

\newcommand{\NN}{\mathbb{N}}

\newcommand{\Cc}{\mathcal{C}}
\newcommand{\Dc}{\mathcal{D}}
\newcommand{\Ec}{\mathcal{E}}

\newcommand{\Ac}{\mathcal{A}}
\newcommand{\Bc}{\mathcal{B}}

\newcommand{\Psf}{\mathsf{P}}
\newcommand{\Qsf}{\mathsf{Q}}

\renewcommand{\Mc}{\mathcal{M}}

\newcommand{\RCA}{\mathsf{RCA}}
\newcommand{\RCOLOR}{\mathsf{RCOLOR}}
\newcommand{\COLOR}{\mathsf{COLOR}}
\newcommand{\AMT}{\mathsf{AMT}}
\newcommand{\WKL}{\mathsf{WKL}}
\newcommand{\CADS}{\mathsf{CADS}}
\renewcommand{\POS}{\mathsf{POS}}
\newcommand{\RWKL}{\mathsf{RWKL}}
\newcommand{\ADS}{\mathsf{ADS}}
\newcommand{\SADS}{\mathsf{SADS}}
\newcommand{\COH}{\mathsf{COH}}
\newcommand{\DNR}{\mathsf{DNC}}
\newcommand{\PIOOG}{\Pi^0_1\mathsf{G}}

\newcommand{\ISig}{\mathsf{I}\Sigma^0}
\newcommand{\BSig}{\mathsf{B}\Sigma^0}

\newcommand{\dom}{\operatorname{dom}}
\newcommand{\uh}[0]{{\upharpoonright}}

\newcommand{\bstr}{2^{<\NN}}

\def\qt#1{``#1''}%

\AtBeginDocument{%
   \def\MR#1{}
}

\usepackage{xcolor}

\title{The weakness of typicality}

\author[Astor]{Eric P. Astor}
\address{
(formerly) Department of Mathematics\\
University of Connecticut\\
Storrs, CT 06269\\
USA}
\curraddr{447 Broadway, 2nd Floor \#1213\\
New York, NY 10013\\
USA}
\email{eric.astor@gmail.com}
\urladdr{https://ericastor.info}

\author[Bienvenu]{Laurent Bienvenu}
\address{Laboratoire Bordelais de Recherche en Informatique\\
Université de Bordeaux\\
351, Cours de la Lib\'eration\\
33405 Talence Cedex\\
France}
\email{laurent.bienvenu@computability.fr}
\urladdr{https://www.labri.fr/perso/lbienvenu/}

\author[Dzhafarov]{Damir Dzhafarov}
\address{Department of Mathematics\\
University of Connecticut\\
Storrs, CT 06269\\
USA}
\email{damir@uconn.edu}
\urladdr{https://damirdzhafarov.com}

\author[Patey]{Ludovic Patey}
\address{Institut de Math\'ematiques de Jussieu-Paris Rive Gauche\\
Universit\'e Paris Cit\'e – B\^atiment Sophie Germain\\
8 Place Aur\'elie Nemours, 75205 Paris Cedex 13}
\email{ludovic.patey@computability.fr}
\urladdr{https://ludovicpatey.com}

\author[Shafer]{Paul Shafer}
\address{School of Mathematics\\
University of Leeds\\
Leeds\\
LS2 9JT\\
United Kingdom}
\email{p.e.shafer@leeds.ac.uk}
\urladdr{https://peshafer.github.io}

\author[Solomon]{Reed Solomon}
\address{Department of Mathematics\\
University of Connecticut\\
Storrs, CT 06269\\
USA}
\email{david.solomon@uconn.edu}
\urladdr{https://www2.math.uconn.edu/~solomon/}

\author[Westrick]{Linda Westrick}
\address{Department of Mathematics\\
Penn State University\\
University Park, PA 16803\\
USA}
\email{westrick@psu.edu}
\urladdr{https://lindawestrick.com}

\date{\today}

\begin{document}

\begin{abstract}
Many statements studied in reverse mathematics can be seen as mathematical problems, formulated in terms of instances and solutions. We develop a framework of typicality encompassing measure and genericity, and we classify the reverse mathematics zoo in terms of which problems admit typical solutions. It turns out that even very weak problems do not admit typical solutions.
\end{abstract}

\maketitle

\section{Introduction}
In mathematical practice, it is common to define objects as ``typical'' if they avoid all extraordinary circumstances up to a certain degree of complexity. For instance, a set of points is in general linear position if it satisfies no more linear relations than are necessary; such a set might be said to be typical with respect to linear relations.

In Cantor space ($2^{\omega}$), two notions of typicality currently dominate: Martin-L\"of $n$-randomness, or avoidance of all ``effectively null'' classes as defined by $\Sigma^0_n$ Martin-L\"of tests; and weak Cohen $n$-genericity, or avoidance of all meager classes whose complements are countable intersections of classes described by $\Sigma^0_n$ sets of strings.

We know that each sort of typicality forces certain properties for subsets of~$\omega$. For example, no $n$-generic or $n$-random set is computable (or in fact $\Sigma^0_n$, for obvious reasons), but no $1$-generic computes a DNC function; on the other hand, every $1$-random does compute a DNC function. In fact, increasing the level of randomness can significantly increase the power of a typical oracle. The Rainbow Ramsey Theorem for pairs and $2$-bounded colors states that every $2$-bounded coloring of pairs in $\omega$ has a ``rainbow'', an infinite subdomain on which no color appears twice. There are known computable $2$-bounded colorings with no computable rainbow; however, as shown by Csima and Mileti~\cite{CsimaM2009}, every $2$-random oracle computes rainbows for every computable $2$-bounded coloring.

Applying (the relativized versions of) these results in reverse mathematics, we see that axioms guaranteeing the mere existence of sufficiently typical sets can have useful implications:

\begin{thm}
$\mathsf{RCA}_0+\mathsf{1\text{-}RAN}\vdash\mathsf{DNC}$.
\end{thm}

\begin{thm}
$\mathsf{RCA}_0+\mathsf{2\text{-}RAN}\vdash\mathsf{RRT}^2_2$.
\end{thm}

This raises further questions. Is there a useful consequence of $\mathsf{3\text{-}RAN}$? What about principles guaranteeing the existence of sufficiently generic sets? How far can these implications go, and are there any results that cannot be proven using typicality alone? Fundamentally, we ask: is the existence of strongly typical sets a weak principle, or a strong assumption?

As suggested by our title, we find that typicality principles are weak. Specifically, we establish strict upper bounds on the set of results implied by randomness- and genericity-existence principles.

In Section~\ref{sec:framework}, we begin by constructing a broad theorem, applicable to all notions of typicality in a broad framework. This framework is fairly simple, and is satisfied by both standard typicality notions on $2^{\omega}$. Our theorem is designed for implications between problems and more precisely $\Pi^1_2$-sentences, but our analysis in Section~\ref{sec:variants} also offers a variation for more general second-order statements.

In Sections~\ref{sec:RANapplications} and \ref{sec:GENapplications}, we explore applications to randomness- and genericity-existence principles respectively. We successfully classify most standard reverse-mathematical principles as either consequences of randomness, or not provable from any randomness-existence principle. We conduct a similar analysis for genericity. 

A note to the reader: in this paper we refer to subsets of $2^{\omega}$ as \emph{classes}, in order to distinguish them from points in $2^{\omega}$, which we think of as \emph{sets}.

\section{Framework}\label{sec:framework}

This section is devoted to the development of an abstract framework for studying the existence of typical solutions to mathematical problems. In \Cref{sec:smallness-abstract}, we give an axiomatic definition of smallness, encompassing randomness and genericity. In \Cref{sec:problems}, we define the abstract notion of mathematical problem, in terms of instances and solutions, and compare them with respect to computable and $\omega$-reducibilities. In \Cref{sec:main-theorem}, we define what it means for a problem to be frequently solved, and prove our main theorem, stating that being frequently solved is upward-closed under $\omega$-reducibility.
Last, in \Cref{sec:variants}, we refine our notion of frequently solved problem to obtain three properties, and obtain further separations.

\subsection{Smallness and Typicality}\label[section]{sec:smallness-abstract}
Suppose we have a property of classes $\mathcal{C}\subseteq 2^{\omega}$, called \emph{smallness}, with the following properties:
\begin{enumerate}
\item \label{stmt:nontrivial} $\emptyset$ is small, and $2^{\omega}$ is not.
\item \label{stmt:nonempty} If $\mathcal{B}$ is small and $\Ac \subseteq \Bc$, then $\Ac$ is small.
\item \label{stmt:union} If $\mathcal{B}_i$ is small for all $i\in\omega$, then $\bigcup_{i\in\omega}{\mathcal{B}_i}$ is also small.
\item \label{stmt:joinProjection} Smallness is closed under the operation:
\begin{equation*}
\mathcal{B}(\mathcal{S})=\setbuildc{X}{\setbuildc{Z}{X\oplus Z\in\mathcal{S}}\text{ is not small}};
\end{equation*}
that is, if $\mathcal{S}$ is small, so is $\mathcal{B}(\mathcal{S})$.
\end{enumerate}

Properties~(\ref{stmt:nontrivial})-(\ref{stmt:union}) are those of a proper $\sigma$-ideal;
thus, for most purposes we merely add that our notion of smallness satisfies property~(\ref{stmt:joinProjection}).

We can define a class to be small (in the sense of measure) if it has measure~$0$; property~(\ref{stmt:joinProjection}) then holds as a consequence of Fubini's Theorem. In this case, a small class $\mathcal{C}$ is said to be $n$-small if it has a $\Sigma^0_n$-Martin-L\"of test; that is, if there is a sequence of uniformly $\Sigma^0_n$ classes $\set{W_n}$ such that $\mu(W_n)\le 2^{-n}$ for all $n$, with $\mathcal{C}\subseteq\bigcap_{n<\omega}{W_n}$. More generally, $\mathcal{C}$ is $n$-$X$-small if we replace $\Sigma^0_n$ by $\Sigma^{0,X}_n$ in this definition.

We can also define a class to be small (in the sense of category) if it is meager; property~(\ref{stmt:joinProjection}) is then essentially equivalent to the weak direction of the Kuratowski-Ulam Theorem. In this case, we say that a small class $\mathcal{C}$ is $n$-small if its complement is covered by an intersection of classes described by dense $\Sigma^0_n$ sets of strings. More generally, $\mathcal{C}$ is $n$-$X$-small if we replace $\Sigma^0_n$ by $\Sigma^{0,X}_n$ in this definition. 

A set is \emph{typical} if it avoids all small classes up to some descriptive complexity. Specifically, $T\subseteq\omega$ is $n$-$X$-typical if it avoids all $n$-$X$-small classes. We note that the class of all $n$-$X$-typical sets, as a countable intersection of co-small classes, will always be co-small; in particular, for all $n$ and all $X$, every non-small class has an $n$-$X$-typical member.

\subsection{Problems and reducibilities}\label[section]{sec:problems}

A \emph{problem} is a relation $\Psf \subseteq 2^\omega \times 2^\omega$. A \emph{$\Psf$-instance} is any element of $\dom \Psf = \{ X : \exists Y (X, Y) \in \Psf \}$, and a \emph{$\Psf$-solution} to $X$ is any element of $\Psf(X) = \{ Y : (X, Y) \in \Psf \}$. Many statements studied in reverse mathematics are of the form
\begin{equation*}
(\forall X)\bracket*{\phi(X)\rightarrow(\exists Y)\bracket*{\psi(X,Y)}},
\end{equation*}
where $\phi(X)$ and $\psi(X,Y)$ are arithmetic formulas. Any such statement can be seen as a problem~$\Psf$
where $\dom \Psf = \{ X : \phi(X) \}$ and $\Psf(X) = \{ Y : \psi(X,Y) \}$. We call them \emph{$\Pi^1_2$-problems}.
There exist multiple computability-theoretic notions of reducibility to compare the strength of problems. The arguably most natural one is computable reducibility:

\begin{defn}
A problem~$\Psf$ is \emph{computably reducible} to $\Qsf$ (written $\Psf \leq_c \Qsf$) if for every~$\Psf$-instance~$X$, there is an $X$-computable $\Qsf$-instance~$\hat X$ such that for every~$\Qsf$-solution~$\hat Y$ to $\hat X$, $\hat Y \oplus X$-computes a $\Psf$-solution to~$X$.
\end{defn}

The notion of computable reducibility was formally introduced by Dzhafarov~\cite{dzhafarovCohesiveAvoidanceStrong2014}, although already present in spirit in many prior works.
A more general notion, which is the computability-theoretic counterpart of implication over~$\RCA_0$, is $\omega$-reducibility, formulated in terms of Turing ideals:

\begin{defn}\ 
\begin{itemize}
    \item[(1)] A \emph{Turing ideal} $\Mc$ is a non-empty collection of sets which is downward-closed under Turing reducibility ($\forall X \in \Mc~ \forall Y \leq_T X,\ Y \in \Mc$) and closed under the effective join ($\forall X, Y \in \Mc, X \oplus Y \in \Mc$).
    \item[(2)] A problem~$\Psf$ \emph{holds} in a Turing ideal~$\Mc$ (written $\Mc \models \Psf$) if for every~$\Psf$-instance~$X \in \Mc$, there is a $\Psf$-solution~$Y \in \Mc$.
    \item[(3)] A problem~$\Psf$ is \emph{$\omega$-reducible} to $\Qsf$ (written $\Psf \leq_\omega \Qsf$) if for every Turing ideal~$\Mc$, if $\Mc \models \Qsf$ then $\Mc \models \Psf$.
\end{itemize}
\end{defn}

Note that $\omega$-reducibility is coarser than computable reducibility, in that if $\Psf \leq_c \Qsf$, then $\Psf \leq_\omega \Qsf$, but the converse does not hold in general. Also note that if $\Psf$ and $\Qsf$ are $\Pi^1_2$-problems such that $\RCA_0 \vdash \Qsf \to \Psf$, then $\Psf \leq_\omega \Qsf$. Most separations in reverse mathematics are actually proofs of non-reduction for $\omega$-reducibility. When restricted to $\Pi^1_2$-problems, $\omega$-reducibility is a particular case of the notion of computable entailment introduced by Shore~\cite{shoreReverseMathematicsPlayground2010}.


\subsection{Main Theorem}\label[section]{sec:main-theorem}

\begin{defn}
A problem $\mathsf{Q}$ is \emph{frequently solved} if, for every $\mathsf{Q}$-instance $X$,
\begin{equation*}
\setbuildc{Z}{(\exists Y\leT X\oplus Z)\bracket*{Y \text{ is a $\Qsf$-solution to } X}}\text{ is not small};
\end{equation*}
that is, each instance has solutions below a non-small class of oracles, and thus below some typical oracle at any complexity.
\end{defn}

In the sense of category, our prototypical examples of frequently solved principles have the form ``For every $X$, there exists a sufficiently $X$-generic set''; in the sense of measure, we use ``For every $X$, there exists a sufficiently $X$-random set''.

By contrast, fixing a problem $\Psf$, we say that
\begin{defn}
A $\Psf$-instance $X$ is \emph{rarely solved} if
\begin{equation*}
\setbuildc{Z}{(\exists Y\leT X\oplus Z)\bracket*{Y \text{ is a $\Psf$-solution to } X}}\text{ is small};
\end{equation*}
that is, $X$ has solutions below only a small class of oracles (even when also provided the instance $X$). In particular, no sufficiently typical oracle can solve $X$.
\end{defn}
Naturally, a principle has a rarely-solved instance if and only if it is not frequently solved.
For our main theorem, we show that being frequently solved is preserved by $\omega$-reducibility. Specifically,

\begin{thm}\label{thm:mainTheorem}
If $\mathsf{P}$ is frequently solved, but $\mathsf{Q}$ has a rarely-solved instance $X$, then $\Qsf \not \leq_\omega \Psf$. 
\end{thm}

In particular, if $\Psf$ and $\Qsf$ are $\Pi^1_2$-problems such that $\Psf$ is frequently solved, but~$\Qsf$ has a rarely-solved instance~$X$, then  $\mathsf{P}$ does not imply $\mathsf{Q}$ over $\mathsf{RCA}_0$.

To construct the Turing ideal witnessing $\Qsf \not \leq_\omega \Psf$, we will begin with a Turing ideal containing only sets computable from $X$; since $X$ is rarely solved, this includes no solutions to the $\Qsf$-instance. We then add solutions to all $\Psf$-instances while avoiding $\mathcal{C}$, the class of all oracles that solve $X$. However, we need to prevent added $\mathsf{P}$-solutions from \emph{jointly} forcing us to add a $\Qsf$-solution to $X$; we do this by instead avoiding any set with respect to which $X$ could become frequently solved. This requires we start with some technical lemmas; we first show that the $\Bc$-operator from smallness axiom (\ref{stmt:joinProjection}) expands $\mathcal{C}$ to include more of these dangers, and then show that iterating this expansion tends towards a fixed point that remains small.

\begin{defn}
Let $\mathcal{C}\subseteq 2^{\omega}$. We say $\mathcal{C}$ is an $X_0$-\emph{computational class} if for any $X\in\mathcal{C}$ and any $Y$ with $X_0\oplus Y\geT X$, $Y\in\mathcal{C}$; that is, if $\mathcal{C}$ is closed upwards under Turing reducibility relative to~$X_0$.
\end{defn}

The following simple lemma states an important property of small $X_0$-computa\-tional classes: they avoid~$X_0$.
As we shall expand our small class by iterating the $\Bc$-operator, proving that the operator preserves being small and $X_0$-computational is a way to ensure that the expanded class still does not contain~$X_0$.

\begin{lemma}\label[lemma]{lem:small-avoids}
If $\Cc$ is a small $X_0$-computational class, then $X_0 \not \in \Cc$.
\end{lemma}
\begin{proof}
Suppose for the contradiction that $X_0 \in \Cc$. Then, since $X_0 \oplus Z \geq_T X_0$ for any set~$Z \in 2^\omega$,
$Z \in \Cc$, so $\Cc = 2^\omega$ which is not small.
\end{proof}

The next lemma shows that the $\Bc$-operator preserves being small and an $X_0$-computational class:

\begin{lemma}\label{lem:BOperator}
If $\mathcal{S}\subseteq 2^{\omega}$ is a small $X_0$-computational class, then
\begin{align*}
\mathcal{B}(\mathcal{S})
&=\setbuildc*{X}{\setbuildc*{Z}{X_0\oplus X\oplus Z\in\mathcal{S}}\text{ is not small}}\\
&=\setbuildc*{X}{\setbuildc*{Z}{X\oplus Z\in\mathcal{S}}\text{ is not small}}.
\end{align*}
is a small $X_0$-computational class with $\mathcal{S}\subseteq\mathcal{B}(\mathcal{S})$.
\end{lemma}

\begin{proof}
First, we note that the two specifications of $\mathcal{B}(\mathcal{S})$ are in fact the same, as $X_0\oplus X\oplus Z\in\mathcal{S}$ if and only if $X\oplus Z\in\mathcal{S}$. (After all, $X_0\oplus X\oplus Z$ is equivalent to $X\oplus Z$ under Turing reducibility relative to $X_0$.)

Consider $X\in\mathcal{S}$. Since $X\oplus Z\geT X$ and $\mathcal{S}$ is a $X_0$-computational class, $X\oplus Z\in\mathcal{S}$ for every $Z\in 2^{\omega}$. Thus, $\setbuildc{Z}{X\oplus Z\in\mathcal{S}} = 2^\omega$ is certainly not small, so $X\in\mathcal{B}(\mathcal{S})$. Since $X$ was arbitrary, $\mathcal{S}\subseteq\mathcal{B}(\mathcal{S})$.

Next, suppose $X\in\mathcal{B}(\mathcal{S})$ and $X_0\oplus Y\geT X$. Clearly, $X_0\oplus Y\oplus Z\geT X\oplus Z$ for any $Z\in 2^{\omega}$. Thus, since $\mathcal{S}$ is an $X_0$-computational class, $X\oplus Z\in\mathcal{S}$ implies $Y\oplus Z\in\mathcal{S}$. This means that $\setbuildc{Z}{Y\oplus Z\in\mathcal{S}}\supseteq\setbuildc{Z}{X\oplus Z\in\mathcal{S}}$ (which is not small), so $Y\in\mathcal{B}(\mathcal{S})$. Therefore, $\mathcal{B}(\mathcal{S})$ is an $X_0$-computational class.

Last, since $\mathcal{S}$ is small, $\mathcal{B}(\mathcal{S})$ is small by property~(\ref{stmt:joinProjection}).
\end{proof}

Given a small $X_0$-computational class $\mathcal{C}$, let $\mathcal{B}^0_{\mathcal{C}}=\mathcal{C}$, and define $\mathcal{B}^{n+1}_{\mathcal{C}}=\mathcal{B}(\mathcal{B}^n_{\mathcal{C}})$. By Lemma~\ref{lem:BOperator}, these classes form an increasing chain of small $X_0$-computational classes
\begin{equation*}
\mathcal{C}=\mathcal{B}^0_{\mathcal{C}}\subseteq\mathcal{B}^1_{\mathcal{C}}\subseteq\mathcal{B}^2_{\mathcal{C}}\subseteq\cdots
\end{equation*}
Let $\mathcal{B}^{\omega}_{\mathcal{C}}=\bigcup_{n\in\omega}{\mathcal{B}^n_{\mathcal{C}}}$. We show that $\mathcal{B}^{\omega}_{\mathcal{C}}$ is a small $X_0$-computational class and a fixed point for the $\mathcal{B}$ operator.

\begin{lemma}\label[lemma]{lem:smallness-closure-B}
If $\mathcal{C}$ is a small $X_0$-computational class, so is $\mathcal{B}^{\omega}_{\mathcal{C}}$.
\end{lemma}
\begin{proof}
Since $\mathcal{B}^{\omega}_{\mathcal{C}}$ is an increasing countable union of small classes, it must itself be small. Let $X_0\oplus Y\geT X$ for some $X\in\mathcal{B}^{\omega}_{\mathcal{C}}$. Since $X\in\mathcal{B}^n_{\mathcal{C}}$ for some $n$, and $\mathcal{B}^n_{\mathcal{C}}$ is an $X_0$-computational class, $Y\in\mathcal{B}^n_{\mathcal{C}}\subseteq\mathcal{B}^{\omega}_{\mathcal{C}}$. Therefore, $\mathcal{B}^{\omega}_{\mathcal{C}}$ is an $X_0$-computational class.
\end{proof}

\begin{lemma}\label{lem:fixedPoint}
For any small $X_0$-computational class $\mathcal{C}$, $\mathcal{B}(\mathcal{B}^{\omega}_{\mathcal{C}})=\mathcal{B}^{\omega}_{\mathcal{C}}$.
\end{lemma}

\begin{proof}
Suppose $X\in\mathcal{B}(\mathcal{B}^{\omega}_{\mathcal{C}})$; that is, $\setbuildc{Z}{X_0\oplus X\oplus Z\in\mathcal{B}^{\omega}_{\mathcal{C}}}$ is not small. Let $\chi_n=\setbuildc{Z}{X_0\oplus X\oplus Z\in\mathcal{B}^n_{\mathcal{C}}}$; clearly, $\chi_n\subseteq\chi_{n+1}$, and
\[
\bigcup_{n\in\omega}{\chi_n}=\setbuildc{Z}{X_0\oplus X\oplus Z\in\mathcal{B}^{\omega}_{\mathcal{C}}}.
\]
Since this increasing countable union is not small, there must be some $\chi_n$ that is not small, so $X\in\mathcal{B}(\mathcal{B}^n_{\mathcal{C}})=\mathcal{B}^{n+1}_{\mathcal{C}}\subseteq\mathcal{B}^{\omega}_{\mathcal{C}}$. We have thus established that $\mathcal{B}(\mathcal{B}^{\omega}_{\mathcal{C}}) \subseteq \mathcal{B}^{\omega}_{\mathcal{C}}$. The reverse inclusion follows directly from Lemma~\ref{lem:smallness-closure-B} and Lemma~\ref{lem:BOperator}.
\end{proof}

Towards our main theorem, we next show that for any $\mathsf{P}$-instance, we can always add a solution without adding any elements of a small class.

\begin{lemma}\label{lem:classAvoidance}
Let $\mathsf{P}$ be a frequently solved problem, and let $\mathcal{C}$ be a small $X_0$-computational class. Fix some $A\subseteq\omega$ such that $A\not\in\mathcal{B}^{\omega}_{\mathcal{C}}$.

If $X \leT A$ is a $\mathsf{P}$-instance, then $X$ has a $\mathsf{P}$-solution $Y$ such that $A\oplus Y\not\in\mathcal{B}^{\omega}_{\mathcal{C}}$.
\end{lemma}

\begin{proof}
By Lemma~\ref{lem:fixedPoint}, $A\not\in\mathcal{B}(\mathcal{B}^{\omega}_{\mathcal{C}})$. Thus, $\setbuildc{Z}{X_0\oplus A\oplus Z\in\mathcal{B}^{\omega}_{\mathcal{C}}}$ is small.

Since $\mathsf{P}$ is frequently solved, the class 
$$\setbuildc{Z}{X \oplus Z \text{ computes a $\Psf$-solution to } X }$$ 
is not small. Therefore, there exist sets $Z$ and $Y$ such that $X \oplus Z\geT Y$, $Y$ is a $\mathsf{P}$-solution to $X$, and $X_0\oplus A\oplus Z\not\in\mathcal{B}^{\omega}_{\mathcal{C}}$. Since $A \geT X$ and $\mathcal{B}^{\omega}_{\mathcal{C}}$ is an $X_0$-computational class, $A\oplus Y\not\in\mathcal{B}^{\omega}_{\mathcal{C}}$.
\end{proof}

\begin{thm}\label{thm:classAvoidance}
Let $\mathsf{P}$ be a frequently-solved problem, and let $\mathcal{C}$ be a small $X_0$-computational class. There exists a Turing ideal~$\Mc \models \mathsf{P}$ that contains $X_0$, but includes no sets from $\mathcal{C}$.
\end{thm}

\begin{proof}
We start with the Turing ideal containing only the sets computable from $X_0$. 
By \Cref{lem:small-avoids} and \Cref{lem:smallness-closure-B}, $X_0 \not \in \Bc_\Cc^\omega$.
Iterating Lemma~\ref{lem:classAvoidance}, we add solutions to $\mathsf{P}$-instances in the standard dovetailing manner while avoiding all computations of elements of $\mathcal{B}^{\omega}_{\mathcal{C}}$. Since $\mathcal{C}\subseteq\mathcal{B}^{\omega}_{\mathcal{C}}$, the resulting Turing ideal~$\Mc$ will satisfy $\mathcal{M}\cap\mathcal{C}=\emptyset$.
\end{proof}

Finally, we apply this to the class of oracles that solve a rarely-solved $\mathsf{Q}$-instance, and obtain our main result.

\begin{proof}[Proof of Theorem~\ref{thm:mainTheorem}]
Fix some rarely-solved instance $X_0$ of $\mathsf{Q}$, and let
\[
\mathcal{C}=\setbuildc{Z}{X_0\oplus Z \text{ computes a $\Qsf$-solution to } X_0 }.
\]
$\mathcal{C}$ is an $X_0$-computational class by definition, and it is small by assumption. By Theorem~\ref{thm:classAvoidance}, there is a Turing ideal  $\mathcal{M}$ such that $\mathcal{M}\models\mathsf{P}$, $\mathcal{M}$ includes $X_0$, and $\mathcal{M}$ includes no sets from $\mathcal{C}$. In particular, $\mathcal{M}$ contains no $\mathsf{Q}$-solutions to $X_0$, so $\mathcal{M}\nmodels\mathsf{Q}$.
\end{proof}

\begin{rem}
Note that the proof of \Cref{thm:mainTheorem} used the axiom (\ref{stmt:union}) of smallness only for increasing sequences of small classes ($\Bc_0 \subseteq \Bc_1 \subseteq \cdots$). One could therefore restrict the axiom to such sequences. However, as soon as we require our notion of smallness to be an ideal -- which is a reasonable assumption removing pathological behaviors --  the axiom of increasing countable unions implies the axiom of arbitrary countable unions.
\end{rem}

\section{Analysis and Variants}\label{sec:variants}
On inspection, our notions of frequently-solved problems and/or rarely-solved instances leave some room for variation. For instance, we could define three possible properties of a problem $\mathsf{R}$:
\begin{enumerate}
\item[P$_1$:] There is a computable instance $C$ such that only a small class of $X$'s compute a solution to $C$.
\item[P$_2$:] For a co-small class of $Z$'s, $Z$ computes an instance $C_Z$ for which $X\oplus Z$ computes a solution to $C_Z$ for only a small class of $X$'s.
\item[P$_3$:] There is an instance $C$ such that only a small class of $X$'s compute a solution to $C$.
\end{enumerate}

Phrased in terms of our definitions from above, we recognize each as speaking about rarely-solved instances; $\mathsf{R}$ has P$_1$ if there is a computable rarely-solved instance, P$_2$ if most $Z$'s compute a $Z$-rarely-solved instance, and P$_3$ if there is any rarely-solved instance at all.

\begin{prop}\label[proposition]{prop:pi-linearly-ordered}
$\textup{P}_1\Rightarrow\textup{P}_2\Rightarrow\textup{P}_3$, and each of these implications is strict, as witnessed by $\Pi^1_2$-problems.
\end{prop}
\begin{proof}
The implications are clear from the observations above. To see that P$_2$ does not imply P$_1$, consider the principle $\mathsf{R}$ stating ``for all $X$, if $X$ is not computable, then $X'$ exists.''  Taking small to mean null, $\mathsf{R}$ has property P$_2$ but not property P$_1$. Similarly for P$_2$ versus P$_3$, consider the principle $\mathsf{R}$ stating ``for all $X$, if $X=0'$, then $X'$ exists.''  Again taking small to mean null, $\mathsf{R}$ has property P$_3$ but not property P$_2$.
\end{proof}

As it turns out, most natural problems studied in reverse mathematics satisfying property P$_3$ also satisfy P$_1$, that is, the rarely-solved instance is witnessed by a computable instance.  On its face, property~P$_2$ seems to have few advantages, and the property itself can be argued to be less natural in this presentation. However, it has one significant advantage: under common assumptions, we can reformulate P$_2$ to apply to non-$\Pi^1_2$ second-order principles while maintaining the resulting $\omega$-model separation.  For this purpose, given sets $\set{X_i}_{i<\omega}$, we write $\mathcal{I}(\set{X_i}_{i<\omega})$ for the Turing ideal they generate, and (by an abuse of notation) also for the $\omega$-model with this ideal as its second-order part.

\begin{defn}

A statement $\mathsf{R}$ has the \emph{no-typical-model} (NTM) property if the class of $X=\bigoplus_{i<\omega}{X_i}$ such that $\mathcal{I}(\set{X_i}_{i<\omega})\vDash\mathsf{R}$ is small.
\end{defn}

We obtain our main separation as applied to the NTM property:
\begin{obs}\label{obs:NTMreduce}
If $\mathsf{P}$ does not have the NTM property, but $\mathsf{Q}$ does, then $\mathsf{Q} \nleq_\omega \mathsf{P}$. In particular, $\mathsf{P}$ does not imply $\mathsf{Q}$ over $\mathsf{RCA}_0$. 

Indeed, letting $\Ac$ and $\Bc$ be the classes of all $X=\bigoplus_{i<\omega}{X_i}$ such that $\mathcal{I}(\set{X_i}_{i<\omega})\vDash\mathsf{P}$ and $\mathcal{I}(\set{X_i}_{i<\omega})\vDash\mathsf{Q}$, respectively, $\Bc$ is small but not $\Ac$, so by Property (\ref{stmt:nonempty}) of smallness, $\Ac \not \subseteq \Bc$. In particular, there  is some~$X=\bigoplus_{i<\omega}{X_i} \in \Ac \setminus \Bc$. By definition of $\Ac$ and $\Bc$, $\mathcal{I}(\set{X_i}_{i<\omega})\vDash\mathsf{P}$ but  $\mathcal{I}(\set{X_i}_{i<\omega})\not\models \mathsf{Q}$.
\end{obs} 

Under minor assumptions on our notion of smallness, we can show that the NTM property coincides with our property P$_2$.

\begin{defn}
A \emph{tailclass} is a class~$\Cc \subseteq 2^\omega$ such that if $X \in \Cc$ and $Y =^* X$, then $Y \in \Cc$, where $=^*$ means \qt{equals up to finite changes}.
A notion of smallness is \emph{tail-trivial} if every Borel tailclass is either small or co-small.
\end{defn}

We should note that such 0-1 laws exist for both measure and category. For measure, it is a standard computability-theoretic application of Kolmogorov's 0-1 Law, while for category, this is a classical result of descriptive set theory (often called a topological 0-1 law) as applied to Cantor space. Note that all our applications of tail-triviality in this article will be on particular tailclasses which are closed under Turing equivalence : if $X \in \Cc$ and $X \equiv_T Y$, then $Y \in \Cc$. 

\begin{defn}
A notion of smallness is \emph{permutation invariant} if it is closed under permuting bits.  That is, if $f : \omega \to \omega$ is a permutation and $\mathcal{S}$ is small, then $\setbuildc{X \circ f}{X \in \mathcal{S}}$ is also small.  A notion of smallness is $\emph{computably permutation invariant}$ if it is closed under computable permutations.
\end{defn}

Measure and category both yield permutation invariant notions of smallness.  The main use of permutation invariance is that it lets us rearrange the columns of the sets in a small (co-small) set and retain smallness (co-smallness).  For example, if $\mathcal{S}$ is small for a computably permutation invariant notion of smallness, then $\setbuildc{Y \oplus X}{X \oplus Y \in \mathcal{S}}$ is also small.  Permutation invariant notions of smallness also let us apply property~(\ref{stmt:joinProjection}) of smallness to partitions other than the one given by the usual join operation.

\begin{thm}\label[theorem]{thm:ntm-characterization}
Assume that our notion of smallness is tail-trivial and computably permutation invariant.
A $\Pi^1_2$-problem has the NTM property if and only if it has property \textup{P}$_2$.
\end{thm}

\begin{proof}
First suppose that $\mathsf{P}$ has property~P$_2$.  Let $\mathcal{S}$ be the class of all $X \oplus Y$ such that $X$ computes a $\mathsf{P}$-instance to which $X \oplus Y$ does not compute a solution.  Notice that $\mathcal{S}$ is a tailclass, and it is Borel because $\mathsf{P}$ is a $\Pi^1_2$-problem.  Property~P$_2$ says that $\mathcal{B}(\mathcal{S}) = \setbuildc{X}{\setbuildc{Y}{X\oplus Y\in\mathcal{S}}\text{ is not small}}$ is co-small.  Therefore $\mathcal{S}$ is not small by property~(\ref{stmt:joinProjection}) of smallness and so is co-small by tail-triviality.  Rearrange columns so that sets of the form $X \oplus Y$ become sets of the form $Z = \bigoplus_{i<\omega}{Z_i}$, where $X$ corresponds to $Z_0$ and $Y$ corresponds to $\bigoplus_{i > 0}{Z_i}$.  The resulting class $\mathcal{R}$ of all $Z = \bigoplus_{i<\omega}{Z_i}$ such that $Z_0$ computes a $\mathsf{P}$-instance to which $Z$ does not compute a solution is co-small by computable invariance.  If $Z = \bigoplus_{i<\omega}{Z_i}$ is in $\mathcal{R}$, then $\mathcal{I}(\set{Z_i}_{i<\omega})\nvDash\mathsf{P}$, so $\mathsf{P}$ has the NTM property.

Now suppose that $\mathsf{P}$ does not have property~P$_2$.  Let $\mathcal{T}$ be the class of all $X \oplus Y$ such that every $X$-computable $\mathsf{P}$-instance has an $(X \oplus Y)$-computable solution.  The failure of property~P$_2$ means that $\mathcal{B}(\mathcal{T}) = \setbuildc{X}{\setbuildc{Y}{X \oplus Y \in \mathcal{T}} \text{ is not small}}$ is not small.  Thus $\mathcal{T}$ is not small by property~(\ref{stmt:joinProjection}) of smallness.  Now let $\mathcal{R}  = \setbuildc{(X \oplus Y) \oplus W}{X \oplus Y \in \mathcal{T}}$.  Clearly $$\mathcal{B}(\mathcal{R}) = \setbuildc{X \oplus Y}{\setbuildc{W}{(X \oplus Y) \oplus W \in \mathcal{R}}\text{ is not small}} = \mathcal{T}$$ is not small, so again $\mathcal{R}$ is not small by property~(\ref{stmt:joinProjection}) of smallness.  The classes~$\mathcal{T}$ and $\mathcal{R}$ are tailclasses, and they are Borel because $\mathsf{P}$ is a $\Pi^1_2$-problem.  Thus they are both co-small by tail-triviality.  The class $\mathcal{R}$ consists of all $(X \oplus Y) \oplus W$ where every  $X$-computable $\mathsf{P}$-instance has an $(X \oplus Y)$-computable solution.  Given $n$, rearrange columns so that sets of the form $(X \oplus Y) \oplus W$ become sets of the form $Z = \bigoplus_{i<\omega}{Z_i}$, where $X$ corresponds to $\bigoplus_{i\le n}{Z_i}$, $Y$ corresponds to $Z_{n+1}$, and $W$ corresponds to $\bigoplus_{i> n+1}{Z_i}$.  Then the class $\mathcal{R}_n$ of all $Z = \bigoplus_{i<\omega}{Z_i}$ where every ($\bigoplus_{i\le n}{Z_i}$)-computable $\mathsf{P}$-instance has a ($\bigoplus_{i\le n+1}{Z_i}$)-computable solution is co-small by computable invariance.  The class $\bigcap_{n < \omega} \mathcal{R}_n$ is therefore co-small, and $\mathcal{I}(\set{Z_i}_{i<\omega}) \vDash \mathsf{P}$ for every $Z$ in the class.  Thus $\mathsf{P}$ does not have the NTM property.
\end{proof}

When showing that property \textup{P}$_2$ implies the NTM property in the first part of the proof of \Cref{thm:ntm-characterization}, the computable invariance assumption is a matter of convenience as it allows us to be agnostic about the encoding of $\bigoplus_{i<\omega}{Z_i}$.  The assumption may be dropped if one commits to the encoding $\bigoplus_{i<\omega}{Z_i} = Z_0 \oplus (\bigoplus_{i > 0}Z_i) = Z_0 \oplus (Z_1 \oplus (Z_2 \oplus \cdots)\cdots)$.  Conversely, the computable invariance assumption seems essential to our method for showing that if property \textup{P}$_2$ fails, then so does the NTM property.

In terms of properties P$_1$, P$_2$, P$_3$, our \Cref{thm:mainTheorem} says that no statement with property P$_3$ can $\omega$-reduce to a statement without property P$_3$.  \Cref{thm:ntm-characterization} implies that \Cref{thm:mainTheorem} also applies to property P$_2$ provided that the problems in question are $\Pi^1_2$ and the notion of smallness is tail-trivial and computably permutation invariant.

\begin{cor}\label[corollary]{cor:pi-not-pi-separation}
Assume that the notion of smallness is tail-trivial and computably permutation invariant.  Let $\Psf$ and $\Qsf$ be $\Pi^1_2$-problems.
If $\Qsf$ has property \textup{P}$_2$ but $\Psf$ does not, then $\Qsf \nleq_\omega \Psf$.
\end{cor}
\begin{proof}
In this situation, Property P$_2$ is equivalent to the NTM property by \Cref{thm:ntm-characterization}.  Thus $\Qsf \nleq_\omega \Psf$ by \Cref{obs:NTMreduce}.
\end{proof}

Again, \Cref{thm:mainTheorem} implies that property P$_3$ is closed upward under $\omega$-reducibility:  if $\Qsf \leq_\omega \Psf$ and $\Qsf$ has property P$_3$, then so does $\Psf$.  Likewise, \Cref{cor:pi-not-pi-separation} implies that property P$_2$ is also closed upward under $\omega$-reducibility, so long as the notion of smallness is tail-trivial and computably permutation invariant and the problems are $\Pi^1_2$.  The situation for property P$_1$ is slightly different because it only concerns computable instances.

\begin{prop}\label[proposition]{prop:p1-not-p1-separation}
If $\mathsf{P}$ does not have property P$_1$, but $\mathsf{Q}$ does, then $\Qsf \not \leq_c \Psf$.
\end{prop}
\begin{proof}
Let $X$ be the computable $\Qsf$-instance such that the class $\Cc$ of sets computing a $\Qsf$-solution is small.
Let $\hat X$ be an $X$-computable $\Psf$-instance. In particular, $\hat X$ is computable.
Since $\Psf$ does not have property P${}_1$, the class~$\Dc$ of sets computing a $\Psf$-solution to~$\hat X$ is not small. By axiom (2) of smallness, $\Dc \nsubseteq \Cc$, so there is a $\Psf$-solution~$\hat Y$ to $\hat X$ such that $\hat Y$ does not compute a $\Qsf$-solution to~$X$. In particular, $\hat Y \oplus X$ does not compute a $\Qsf$-solution to~$X$, so $\Qsf \not \leq_c \Psf$.
\end{proof}

Many principles studied in reverse mathematics have a computability-theoretic counterpart, using the notion of \emph{bounding degree}.

\begin{defn}
Let $\Psf$ be a problem. A set~$A$ is \emph{$\Psf$-bounding} over~$B$ (written $A \gg_\Psf B$) if 
for every $B$-computable $\Psf$-instance~$X$, there is an $A \oplus B$-computable $\Psf$-solution~$Y$.
\end{defn}

When $B = \emptyset$, we simply say $A$ is \emph{$\Psf$-bounding}, written $A \gg_\Psf \emptyset$.

\begin{thm}\label[theorem]{thm:p1-degrees-bounding}
Assume that our notion of smallness is tail-trivial.
A $\Pi^1_2$-problem $\Psf$ has property \textup{P}${}_1$ if and only if the class of $A$'s that are  $\Psf$-bounding is small.
\end{thm}
\begin{proof}
Suppose first $\Psf$ has property P${}_1$, and let $X$ be the computable $\Psf$-instance witnessing it, i.e., only a small class of $A$'s compute a solution to~$X$. A fortiori, only a small class of $A$'s can compute a solution to all computable $\Psf$-instances, i.e., only a small class of $A$'s can be $\Psf$-bounding. 

Suppose now that the class of $\Psf$-bounding sets is small. Fix a countable enumeration $(X_e)_{e \in \omega}$ of all computable $\Psf$-instances. For every~$e \in \omega$, let $\Cc_e$ be the class of all sets~$A$ which do not compute $\Psf$-solutions to~$X_e$. The fact that the class of $\Psf$-bounding sets is small means that the class $\bigcup_{e \in \omega} \Cc_e$ is co-small. Since smallness is closed under arbitrary countable unions, there is some~$e \in \omega$ such that $\Cc_e$ is not small. Note that $\Cc_e$ is a tailclass. Moreover, $\Cc_e$ is Borel since $\Psf$ is a $\Pi^1_2$-problem, so $\Cc_e$ is co-small. In other words, the $\Psf$-instance $X_e$ witnesses that $\Psf$ has property P${}_1$.
\end{proof}

\section{Applications to Randomness}\label{sec:RANapplications}

The goal of this section is to classify the reverse mathematics zoo in terms of which problems have property P${}_i$ for measure, for $i \in \{1,2,3\}$. Here by `measure' we mean the Lebesgue measure (a.k.a uniform measure) on $2^\omega$ and we denote it by~$\mu$. It can be defined as being the unique Borel probability measure such that for every $\sigma \in 2^{<\omega}$, $\mu \{X : X \upharpoonright |\sigma| = \sigma\} = 2^{-|\sigma|}$. 

As we shall see, most of the statements have property P${}_1$ (and therefore properties P${}_2$ and P${}_3$ by \Cref{prop:pi-linearly-ordered}). By \Cref{prop:p1-not-p1-separation}, it suffices to consider the weakest principles in the zoo having property P${}_1$ and computable-reducibility. All the known principles from the zoo are known to either follow from randomness, in which case they do not even have property P${}_3$, or are above one of the following principles for computable reducibility:
\begin{itemize}
    \item A Ramsey-type graph coloring principle called $\RCOLOR_2$ (\Cref{sec:ran-rcolor2})
    \item The Atomic Model Theorem, $\AMT$ (\Cref{sec:ran-amt})
    \item Cohesiveness principles, $\COH$ and $\CADS$ (\Cref{sec:ran-coh})
    \item The positive measure domination principle, $\POS$ (\Cref{sec:ran-pos})
\end{itemize}
In the remainder of this section, we shall prove that each of the principles above has property P${}_1$.

\subsection{Ramsey-type graph coloring principles}\label[section]{sec:ran-rcolor2}

Flood~\cite{floodReverseMathematicsRamseytype2012} introduced the \emph{Ramsey-type weak K\"onig's lemma} ($\RWKL$), a restriction of weak K\"onig's lemma where a solution consists of infinitely many bits of information about a path. The approach was then systematized by Bienvenu, Patey and Shafer~\cite{bienvenuLogicalStrengthsPartial2017} by considering Ramsey-type versions of multiple principles related to weak K\"onig's lemma. The weakest of these statements is a Ramsey-type version of a coloring principle about graphs:

\begin{defn}
Let $k \in \omega$ and $G = (V, E)$ be a graph.
\begin{itemize}
    \item A \emph{$k$-coloring} is a function $f : V \to k$ such that for every $x, y \in V$, if $(x, y) \in E$ then $f(x) \neq f(y)$.
    \item $G$ is \emph{$k$-colorable} if it has a $k$-coloring, and $G$ is \emph{locally $k$-colorable} if every finite subgraph is $k$-colorable.
    \item A set $H \subseteq V$ is \emph{$k$-homogeneous} for $G$ if every finite $V_0 \subseteq V$ induces a subgraph that is $k$-colorable by a coloring that colors every vertex in~$V_0 \cap H$ color~0.
\end{itemize}
\end{defn}

Let $\COLOR_k$ be the statement \qt{every locally $k$-colorable graph is $k$-colorable}. Hirst~\cite{hirstReverseMathematicsMarriage2015} proved that $\COLOR_k$ is equivalent to $\WKL_0$ over~$\RCA_0$ for every $k \geq 2$. The following statement is the Ramsey-type version of $\COLOR_k$: 

\begin{stmt}
$\RCOLOR_k$ is the statement \qt{for every infinite, locally $k$-colorable graph $G = (V, E)$, there is an infinite set $H \subseteq V$ that is $k$-homogeneous for~$G$.}
\end{stmt}

Bienvenu, Patey and Shafer~\cite{bienvenuLogicalStrengthsPartial2017} proved that $\RCOLOR_k$ is equivalent to $\RWKL$ over~$\RCA_0$ for every $k \geq 3$. On the other hand, it is still unknown whether $\RCOLOR_2$ is strictly weaker.
They proved~\cite[Theorem 6.10]{bienvenuLogicalStrengthsPartial2017} in particular that $\RCOLOR_2$ has the P${}_1$ property for measure:

\begin{thm}[Bienvenu, Patey and Shafer~\cite{bienvenuLogicalStrengthsPartial2017}]
There is a computable bipartite graph $G = (\omega, E)$ such that the measure of the set of  oracles that enumerate homogeneous sets for $G$ is 0.
\end{thm}

\subsection{Atomic model theorem}\label[section]{sec:ran-amt}

The atomic model theorem is a classical statement of model theory, stating that a theory is atomic if and only if it admits an atomic model.
In what follows, every theory~$T$ is \emph{deductively closed} (for every formula~$\varphi$ such that $T \vdash \varphi$, $\varphi \in T$), \emph{complete} ($T \vdash \varphi$ or $T \vdash \neg \varphi$ for every sentence $\varphi$) and \emph{consistent} (there is no formula $\varphi$ such that $T \vdash \varphi$ and $T \vdash \neg \varphi$).

\begin{defn}Let $T$ be a theory.
\begin{itemize}
    \item A formula $\varphi(x_1, \dots, x_n)$ is an \emph{atom} of~$T$ if for each formula $\psi(x_1, \dots, x_n)$, either $T \vdash \varphi \to \psi$ or $T \vdash \varphi \to \neg \psi$, but not both. 
    \item The theory~$T$ is \emph{atomic} if for every formula $\psi(x_1, \dots, x_n)$ consistent with $T$, there is an atom $\varphi(x_1, \dots, x_n)$ of~$T$ such that $T \vdash \varphi \to \psi$.
    \item A model~$\Mc$ of~$T$ is \emph{atomic} if every $n$-tuple from~$\Mc$ satisfies an atom of~$T$.
\end{itemize}
\end{defn}

\begin{stmt}
$\AMT$ is the statement \qt{If a theory is atomic, then it admits an atomic model.}
\end{stmt}

The atomic model theorem was studied from a reverse-mathematical viewpoint by Hirschfeldt, Shore and Slaman~\cite{hirschfeldtAtomicModelTheorem2009}, who proved that it follows from the ascending/descending sequence principle restricted to linear orders of type $\omega+\omega^*$ ($\SADS$).
Conidis~\cite{conidisClassifyingModeltheoreticProperties2008} studied degrees bounding various model-theoretic properties, and proved in particular the following theorem. We say that a function $f : \omega \to \omega$ \emph{dominates} $g : \omega \to \omega$ (written $f \geq g$) if $f(x) \geq g(x)$ for all but finitely many~$x \in \omega$.

\begin{thm}[{Conidis~\cite[Theorem 3.9]{conidisClassifyingModeltheoreticProperties2008}}]\label[theorem]{thm:amt-escape}
For every $\Delta^0_2$ function $f : \omega \to \omega$, there is a computable atomic theory~$T$
such that every atomic model~$\Mc$ of~$T$ computes a function $g : \omega \to \omega$ not dominated by~$f$.
\end{thm}

Note that the original formulation is slightly different, as Conidis introduced the \emph{isolated path property} corresponding to degrees bounding the Atomic Model Theorem, and the \emph{escape property} corresponding to degrees relative to which $\emptyset'$ is non-high.

Dobrinen and Simpson~\cite{dobrinenAlmostEverywhereDomination2004} introduced the following notions of almost-everywhere domination.

\begin{defn}\label[definition]{def:aed}\ 
\begin{itemize}
    \item A set~$A$ is \emph{almost everywhere dominating} (a.e.d) if
    $$
    \mu \{ X : \forall g \leq_T X ~ \exists f \leq_T A, f \geq g\} = 1
    $$
    \item A set~$A$ is \emph{uniformly almost everywhere dominating} (u.a.e.d) if there is a function $f \leq_T A$ such that
    $$
    \mu \{ X : \forall g \leq_T X, f \geq g\} = 1
    $$
\end{itemize}
\end{defn}

Building up on the work of Binns et al~\cite{binns2006conjecture} and Kjos-Hanssen~\cite{Kjos-Hanssen2007low}, Kjos-Hanssen, Miller and Solomon~\cite{Kjos-Hanssen2012lowness} proved that the two notions coincide, answering positively a conjecture of Dobrinen and Simpson~\cite{dobrinenAlmostEverywhereDomination2004}.
Formulated in this framework, Kurtz~\cite{kurtzRandomnessGenericityDegrees1982} proved that $\emptyset'$ is uniformly almost everywhere dominating, from which we deduce the following theorem:

\begin{thm}
$\AMT$ has property \textup{P}${}_1$ for measure.
\end{thm}
\begin{proof}
Let $f : \omega \to \omega$ be a $\Delta^0_2$ function witnessing that $\emptyset'$ is uniformly almost everywhere dominating.
By \Cref{thm:amt-escape}, there is a computable atomic theory~$T$
such that every atomic model~$\Mc$ of~$T$ computes a function $g : \omega \to \omega$ not dominated by~$f$.
Then, the measure of oracles computing an atomic model of~$T$ is~0.
\end{proof}


\subsection{Cohesiveness principles}\label[section]{sec:ran-coh}

Cohesiveness is closely related to the notion of \emph{maximal set}, which was introduced by Friedberg~\cite{friedberg1958three} in the search for a c.e.\ set of incomplete degree.

\begin{defn}
An infinite set $C$ is \emph{cohesive} for a sequence of sets $R_0, R_1, \dots$ if for every~$n \in \omega$, either $C \subseteq^* R_n$, or $C \subseteq^* \overline{R}_n$. Here, $\subseteq^*$ denotes inclusion up to finitely many elements.
\end{defn}

Jockusch and Stephan~\cite{jockusch1993cohesive} studied cohesive sets for the sequences of c.e., computable and primitive recursive sets. Cholak, Jockusch and Slaman~\cite{cholakStrengthRamseyTheorem2001} introduced the cohesiveness principle as a statement in reverse mathematics in order to divide the study of Ramsey's theorem for pairs into its stable and cohesive parts.

\begin{stmt}
$\COH$ is the statement \qt{Every countable sequence of sets admits an infinite cohesive set.}
\end{stmt}

Jockusch and Stephan~\cite{jockusch1993cohesive} proved that the sequence of primitive recursive sets forms a universal instance of~$\COH$ and characterized the degrees of its cohesive sets:

\begin{thm}[Jockusch and Stephan~\cite{jockusch1993cohesive}]\label[theorem]{thm:coh-degrees}\ 
\begin{itemize}
    \item For every infinite cohesive set~$C$ for the sequence of primitive recursive sets, $C'$ is of PA degree over~$\emptyset'$.
    \item Let $\vec{R} = R_0, R_1, \dots$ be a uniformly computable sequence of sets. Every degree whose jump is of PA degree over~$\emptyset'$ computes an infinite cohesive set for~$\vec{R}$.
\end{itemize}
\end{thm}

The relation between PA degrees and measure was extensively studied. In particular, Ku\v{c}era~\cite{kuceraMeasurePi0_11985} proved the following theorem:

\begin{thm}[Ku\v{c}era~\cite{kuceraMeasurePi0_11985}]\label[theorem]{thm:pa-measure}
Fix a set~$Z$. Then
$$
\mu \{ X : X \oplus Z \mbox{ is of PA degree over } Z \} = 0
$$
\end{thm}

Letting $Z = \emptyset'$, this shows that the class of $X$ such that $X \oplus \emptyset'$ is of PA degree over~$\emptyset'$ is null. This is however not sufficient since one needs to consider $X'$ instead of $X \oplus \emptyset'$. Thankfully, Kautz~\cite{kautzDegreesRandomSets1991} proved the following theorem:

\begin{thm}[Kautz~\cite{kautzDegreesRandomSets1991}]\label[theorem]{thm:gl-measure}
For every~$n \geq 1$,
$$
\mu \{ X : X^{(n)} \equiv_T X \oplus \emptyset^{(n)} \} = 1
$$
\end{thm}

This enables us to classify $\COH$ with respect to measure:

\begin{thm}
$\COH$ has property \textup{P}${}_1$ for measure.
\end{thm}
\begin{proof}
Let $\vec{R} = R_0, R_1, \dots$ be the sequence of all primitive recursive sets.
By Jockusch and Stephan (\Cref{thm:coh-degrees}), the jump of any infinite cohesive set for~$\vec{R}$ is of PA degree over~$\emptyset'$.
By Kautz (\Cref{thm:gl-measure}) and Ku\v{c}era (\Cref{thm:pa-measure}), 
$$
\mu \{ X : X' \mbox{ is of PA degree over } \emptyset' \} = \mu \{ X : X \oplus \emptyset' \mbox{ is of PA degree over } \emptyset' \} = 0
$$
It follows that the measure of sets which compute an infinite cohesive set for~$\vec{R}$ is~0.
\end{proof}

In their reverse-mathematical study of partial orders and linear orders, Hirschfeldt and Shore~\cite{hirschfeldtCombinatorialPrinciplesWeaker2007} introduced a cohesive counterpart of the Ascending Descending Sequence principle:

\begin{stmt}
$\CADS$ is the statement \qt{Every infinite linear order has a suborder that is either of type $\omega$ or $\omega^*$ or $\omega+\omega^*$.}
\end{stmt}

Hirschfeldt and Shore~\cite{hirschfeldtCombinatorialPrinciplesWeaker2007} proved in particular that $\COH$ implies $\CADS$ over~$\RCA_0$ and that the converse holds over~$\RCA_0 + \BSig_2$, but left the equivalence over~$\RCA_0$ open. The proofs show that $\COH$ and $\CADS$ seen as problems are computably equivalent:

\begin{thm}[Hirschfeldt and Shore~\cite{hirschfeldtCombinatorialPrinciplesWeaker2007}]\label[theorem]{thm:coh-cads}
$\COH \equiv_c \CADS$.
\end{thm}

Thanks to \Cref{prop:p1-not-p1-separation}, we deduce that

\begin{thm}
$\CADS$ has property \textup{P}${}_1$ for measure.
\end{thm}

\subsection{Positive measure domination}\label[section]{sec:ran-pos}

The positive measure domination principle was introduced by Kjos-Hanssen, Miller and Solomon~\cite{Kjos-Hanssen2012lowness} in their study of almost everywhere domination.

\begin{stmt}
$\POS$ is the statement \qt{For every G${}_\delta$ class $\Cc \subseteq 2^\omega$ of positive measure, there is a closed sub-class $\Dc \subseteq \Cc$ of positive measure.}
\end{stmt}

Recall the notion of almost everywhere domination (\Cref{def:aed}).  
Kjos-Hanssen, Miller and Solomon~\cite{Kjos-Hanssen2012lowness} proved the following equivalence:

\begin{thm}[Kjos-Hanssen, Miller and Solomon~\cite{Kjos-Hanssen2012lowness}]\label[theorem]{thm:uaed-pmd}
Let $A$ be a set. The following are equivalent:
\begin{itemize}
    \item[(1)] $A$ is uniformly almost everywhere dominating
    \item[(2)] For every $\Pi^0_2$ class $\Cc \subseteq 2^\omega$ of positive measure, there is a $\Pi^0_1(A)$ class $\Dc \subseteq \Cc$ of positive measure.
\end{itemize}
\end{thm}

Any set~$A$ satisfying (2) is called \emph{positive measure dominating}. Note that this is the computability-theoretic counterpart of the $\POS$ principle. By Martin~\cite{martinClassesRecursivelyEnumerable1966}, every uniformly almost everywhere dominating degree is high, and therefore, by Kautz~\cite{kautzDegreesRandomSets1991} (see \Cref{thm:gl-measure}), the measure of sets $A$ satisfying (2) is~0. We have all the necessary ingredients to prove the following theorem:

\begin{thm}
$\POS$ has property \textup{P}${}_1$ for measure.
\end{thm}
\begin{proof}
By \Cref{thm:p1-degrees-bounding}, it suffices to show that the class of sets satisfying \Cref{thm:uaed-pmd}(2), or equivalently satisfying \Cref{thm:uaed-pmd}(1), has measure~0. By Kautz~\cite{kautzDegreesRandomSets1991} (see \Cref{thm:gl-measure}), the measure of the class of sets of high degree is~0. By Martin's domination theorem~\cite{martinClassesRecursivelyEnumerable1966}, any set satisfying \Cref{thm:uaed-pmd}(1) is of high degree, so we conclude.
\end{proof}

\section{Applications to Genericity}\label{sec:GENapplications}

We now turn to the classification of the reverse mathematics zoo in terms of properties for genericity.
As for randomness, many very weak principles have property P${}_1$ for genericity, but the exact picture is slightly different. In particular, the atomic model theorem ($\AMT$) has property P${}_1$ for randomness, but it does not even have property P${}_3$ for genericity.

We will give in \Cref{sec:rpi12-framework} a general criterion for deciding properties for genericity over a family of $\Pi^1_2$-problems called \emph{restricted $\Pi^1_2$-problems} satisfied by most principles studied in reverse mathematics.

\subsection{Atomic model theorem}

Recall the atomic model theorem from \Cref{sec:ran-amt}, stating that every atomic theory admits an atomic model.
Hirschfeldt, Shore and Slaman~\cite{hirschfeldtAtomicModelTheorem2009} introduced a genericity statement closely related to the standard Henkin constructions of models of atomic theories: We say that a set~$G \in 2^\omega$ \emph{meets} a set~$D \subseteq \bstr$ if $G \uh_n \in D$ for some~$n \in \omega$.

\begin{stmt}
$\Pi^0_1\mathsf{G}$ is the statement \qt{For every uniformly $\Pi^0_1$ collection of sets $(D_n)_{n \in \omega}$, each of which is dense in~$\bstr$, there is a set that meets each of them.}
\end{stmt}

Hirschfeldt, Shore and Slaman~\cite{hirschfeldtAtomicModelTheorem2009} and Conidis~\cite{conidisClassifyingModeltheoreticProperties2008} proved that $\AMT$ and $\PIOOG$ are computably equivalent, and that the equivalence holds over~$\RCA_0 + \ISig_2$. More precisely, $\RCA_0 \vdash \PIOOG \to \AMT$, but the converse does not hold.

\begin{thm}[Hirschfeldt, Shore and Slaman~\cite{hirschfeldtAtomicModelTheorem2009} and Conidis~\cite{conidisClassifyingModeltheoreticProperties2008}]\label[theorem]{thm:amt-pioog}
$\AMT \equiv_c \PIOOG$.
\end{thm}

Clearly, $\PIOOG$ does not have property P${}_3$ for genericity, so neither does $\AMT$ by \Cref{thm:mainTheorem} and \Cref{thm:amt-pioog}. This has to be put in contrast with measure, for which we proved that $\AMT$ has property P${}_1$. Because of this difference, we will need to consider problems computably above $\AMT$ to classify which admit property P${}_1$ for genericity. Thankfully, the next section gives a general framework for deciding such a property.

\subsection{Restricted $\Pi^1_2$-problems}\label[section]{sec:rpi12-framework}

The following class of restricted $\Pi^1_2$ sentences was introduced by Hirschfeldt, Shore and Slaman~\cite{hirschfeldtAtomicModelTheorem2009} in the study of $\Pi^1_1$-conservation theorems for~$\AMT$.

\begin{defn}
A \emph{restricted $\Pi^1_2$-problem} (r$\Pi^1_2$-problem) is a problem of the form
$$
\forall X [\Phi(X) \to \exists Y \Psi(X, Y)]
$$
where $\Phi$ is arithmetic and $\Psi$ is $\Sigma^0_3$.
\end{defn}

Note that \Cref{thm:rp12-criterion-genericity} below will remain true even if we allow $Y$ to be quantified over $\omega^\omega$.
Many problems considered in reverse mathematics are restricted $\Pi^1_2$. We shall actually see that there is a simple criterion to decide whether an r$\Pi^1_2$-problem has one of those properties.

Jockusch and Soare~\cite[Corollary 5.2]{jockuschPi0_1Classes1972} proved the following theorem:

\begin{thm}[Jockusch and Soare~\cite{jockuschPi0_1Classes1972}]\label[theorem]{thm:meager-sigma3}
Let $\Cc \subseteq \omega^\omega$ be a $\Sigma^0_3$ class with no computable member.
Then the class of all sets computing any member of~$\Cc$ is meager.
\end{thm}

Note that for any computable instance~$X$ of an r$\Pi^1_2$-problem~$\Psf$, the class $\Psf(X)$ of its solutions is $\Sigma^0_3$ in the Cantor space. We therefore deduce our general criterion:

\begin{thm}\label[theorem]{thm:rp12-criterion-genericity}
Let $\Psf$ be a r$\Pi^1_2$-problem.
\begin{itemize}
    \item[(1)] $\Psf$ has property \textup{P}${}_1$ for genericity if and only if there is a computable $\Psf$-instance with no computable $\Psf$-solution.
    \item[(2)] $\Psf$ has property \textup{P}${}_2$ for genericity if and only if the class of sets~$A$ computing a $\Psf$-instance with no $A$-computable $\Psf$-solution, is co-meager.
    \item[(3)] $\Psf$ has property \textup{P}${}_3$ for genericity if and only if there is a $\Psf$-instance~$X$ with no $X$-computable $\Psf$-solution.
\end{itemize}
\end{thm}
\begin{proof}
The forward direction for each of~(1), (2) and (3) is immediate. We prove the backward directions.
\begin{itemize}
    \item[(1)] Suppose that there is a computable $\Psf$-instance~$X$ with no computable $\Psf$-solution. The class~$\Psf(X)$ of its solutions being $\Sigma^0_3$ with no computable member, by \Cref{thm:meager-sigma3}, the class of all sets computing any member of~$\Psf(X)$ is meager, so $X$ witnesses that $\Psf$ has property P${}_1$.
    \item[(2)] Suppose that the class of sets~$A$ computing a $\Psf$-instance with no $A$-computable $\Psf$-solutions is co-meager. For any such set~$A$, letting $X$ be the $A$-computable $\Psf$-instance with no $A$-computable $\Psf$-solution, the class $\Psf(X)$ of its $\Psf$-solutions is $\Sigma^0_3(A)$ with no $A$-computable member, so by \Cref{thm:meager-sigma3} relativized to~$A$, the class of all sets $A$-computing any member of~$\Psf(X)$ is meager. It follows that $\Psf$ has property P${}_2$.
    \item[(3)] Suppose that there is a $\Psf$-instance~$X$ with no $X$-computable $\Psf$-solution. The class~$\Psf(X)$ of its solutions being $\Sigma^0_3(X)$ with no $X$-computable member, by \Cref{thm:meager-sigma3} relativized to~$X$, the class of all sets $X$-computing any member of~$\Psf(X)$ is meager, so $X$ witnesses that $\Psf$ has property P${}_3$.
\end{itemize}
\end{proof}

We now list some r$\Pi^1_2$-problems which are minimal for property P${}_1$ for genericity.
\smallskip

In their study of partial and linear orders, Hirschfeldt and Shore~\cite{hirschfeldtCombinatorialPrinciplesWeaker2007} decomposed the Ascending Descending Sequence principle ($\ADS$) into a cohesive ($\CADS$) and a stable ($\SADS$) version. The $\CADS$ principle was already introduced and studied in \Cref{sec:ran-coh} for randomness.

\begin{stmt}
$\SADS$ is the statement \qt{For every infinite linear order of type $\omega+\omega^*$, there is an infinite ascending or descending sequence.}
\end{stmt}

Tennenbaum (see Rosenstein~\cite{rosensteinLinearOrderings1983}) and Denisov (see Goncharov and Nurtazin~\cite{goncharov1973constructive}) constructed a computable linear order of order type $\omega+\omega^*$ with no infinite computable ascending or descending sequence. Moreover, note that $\SADS$ is an r$\Pi^1_2$-problem, so by 
\Cref{thm:rp12-criterion-genericity}(1), $\SADS$ has property P${}_1$ for genericity.
\smallskip

A function $f : \omega \to \omega$ is \emph{diagonally non-$X$-computable} ($X$-DNC) if for every~$e \in \omega$, $f(e) \neq \Phi^X_e(e)$.
The following statement is the reverse-mathematical counterpart of the notion of diagonally non-computable function.

\begin{stmt}
$\DNR$ is the statement \qt{For every set~$X$, there is an $X$-DNC function.}
\end{stmt}

As its name suggests, an immediate diagonal argument gives that there is no computable $\emptyset$-DNC function. Moreover, for every~$X$, the class of $X$-DNC functions is either $\Pi^0_1$ in the Baire space, or $\Pi^0_2$ in the Cantor space, so $\DNR$ is an r$\Pi^1_2$-problem. By \Cref{thm:rp12-criterion-genericity}(1), $\DNR$ has property P${}_1$ for genericity.
\smallskip

Last, recall the Ramsey-type graph coloring principle $\RCOLOR_2$ defined in \Cref{sec:ran-rcolor2}.
Given a computable, infinite locally 2-colorable graph $G = (V, E)$, the set of infinite homogeneous sets for~$G$
is a $\Pi^0_1$ class in the Baire space or a $\Pi^0_2$ class in the Cantor space, so $\RCOLOR_2$ is an r$\Pi^1_2$-problem. Moreover, Bienvenu, Patey and Shafer~\cite{bienvenuLogicalStrengthsPartial2017} constructed a computable instance of~$\RCOLOR_2$ with no computable solution. Again, by \Cref{thm:rp12-criterion-genericity}(1), $\RCOLOR_2$ has property P${}_1$ for genericity.

\subsection{Cohesiveness principles}

The analysis of cohesiveness principles for genericity is almost the same as the one for randomness from \Cref{sec:ran-coh}, except that \Cref{thm:pa-measure} and \Cref{thm:gl-measure} have to be replaced by their genericity counterparts:

The first theorem is the relativization of a theorem of Jockusch and Soare~\cite[Theorem 5.1]{jockuschPi0_1Classes1972}:

\begin{thm}[Jockusch and Soare~\cite{jockuschPi0_1Classes1972}]\label[theorem]{thm:pa-genericity}
Fix a set~$Z$. The following class is meager:
$$
\{ X : X \oplus Z \mbox{ is of PA degree over } Z \}
$$
\end{thm}

The following theorem is considered folklore (see Jockusch~\cite[Lemma 2.6]{jockuschDegreesGenericSets1980}), as it is at the heart of forcing:

\begin{thm}\label[theorem]{thm:gl-genericity}
For every~$n \geq 1$, the following class is co-meager:
$$
\{ X : X^{(n)} \equiv_T X \oplus \emptyset^{(n)} \}
$$
\end{thm}

We are now ready to classify $\COH$ and $\CADS$:

\begin{thm}
$\COH$ and $\CADS$ have property P${}_1$ for genericity.
\end{thm}
\begin{proof}
Let $\vec{R} = R_0, R_1, \dots$ be the sequence of all primitive recursive sets.
By Jockusch and Stephan (\Cref{thm:coh-degrees}), the jump of any infinite cohesive set for~$\vec{R}$ is of PA degree over~$\emptyset'$.
Let $\Cc = \{ X : X' \mbox{ is of PA degree over } \emptyset' \}$,
$\Dc = \{ X : X \oplus \emptyset' \mbox{ is of PA degree over } \emptyset' \}$ and $\Ec = \{ X : X' \equiv_T X \oplus \emptyset' \}$.
Note that $\Cc \subseteq (\Dc \cap \Ec) \cup \Ec^c$.

By \Cref{thm:gl-genericity} and Jockusch and Soare (\Cref{thm:pa-genericity}), 
the classes $\Dc$ and $\Ec^c$ are meager, so $\Cc$ is meager.
It follows that the class of sets which compute an infinite cohesive set for~$\vec{R}$ is meager, so $\COH$ has property P${}_1$ for genericity.

Last, by Hirschfeldt and Shore~\cite{hirschfeldtCombinatorialPrinciplesWeaker2007} (\Cref{thm:coh-cads}), $\COH \equiv_c \CADS$, so by \Cref{prop:p1-not-p1-separation}, $\CADS$ has property P${}_1$ for genericity.
\end{proof}

\subsection{Positive measure domination}

As for cohesiveness, the analysis of the positive measure domination principle ($\POS$) for genericity is very similar to that for randomness from \Cref{sec:ran-pos}. Indeed, the only measure-specific component of the proof was the fact that being of high degree is a measure-0 property. This also holds for genericity by \Cref{thm:gl-genericity}.

\begin{thm}
$\POS$ has property P${}_1$ for genericity.
\end{thm}
\begin{proof}
By \Cref{thm:p1-degrees-bounding}, it suffices to show that the class of sets satisfying \Cref{thm:uaed-pmd}(2), or equivalently satisfying \Cref{thm:uaed-pmd}(1), is meager. By \Cref{thm:gl-genericity} the class of sets of high degree is meager. By Martin's domination theorem~\cite{martinClassesRecursivelyEnumerable1966}, any set satisfying \Cref{thm:uaed-pmd}(1) is of high degree, so we conclude.
\end{proof}

\section{Summary diagrams}

Based on the reverse-mathematical zoo by Astor and Dzhafarov~\cite{rmzoo}, the authors extracted a zoo of $\Pi^1_2$-problems under computable reducibility~\cite{zoodataset}, and classified the problems based on which admits which property for randomness and genericity. As it happens, for both notions of smallness,  every problem in the zoo either has property~P$_1$ (in which case it also has property P$_2$ and P$_3$), or it does not even have property~P$_3$. On the other hand, by \Cref{prop:pi-linearly-ordered}, one can construct artificial $\Pi^1_2$-problems having property~P$_{i+1}$ but not P$_i$ for $i \in \{1,2\}$. 

\begin{openQuestion}
Are there natural $\Pi^1_2$-problems having property P$_2$ but not property P$_1$, or property P$_3$ but not property P$_2$ for randomness? for genericity?
\end{openQuestion}

The notion of natural $\Pi^1_2$-problem is not formally defined, but one could for instance consider any $\Pi^1_2$-sentence appearing in a mathematical textbook at an undergraduate level. 

We conclude with summary diagrams.
\Cref{fig:summary-ran} and \Cref{fig:summary-gen} represent the lower part of the reverse-mathematical zoo under computable reducibility. Double arrows denote strict reductions. Problems in the hatched zone do not have property P${}_3$ for the considered notion of typicality. Problems outside the hatched zone have property P${}_1$.  Every other problem studied in reverse mathematics is computably above one of the problems in the diagram having property P${}_1$.

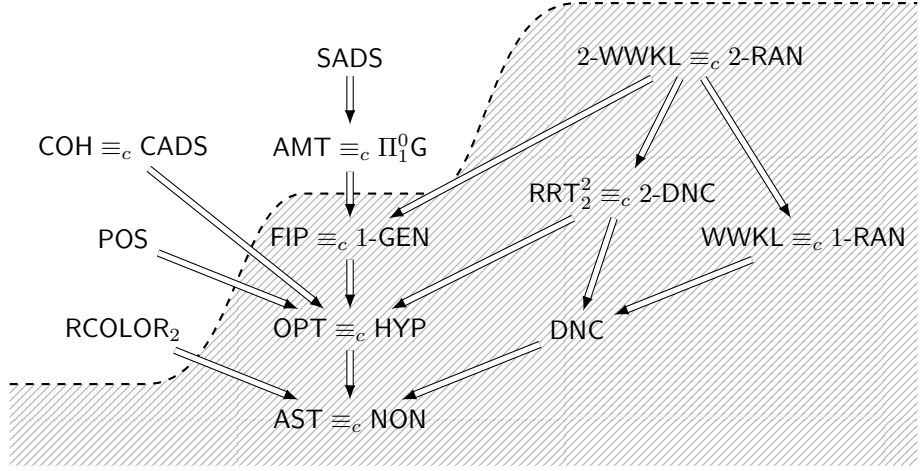
\begin{figure}[htbp]
\begin{center}
\begin{tikzpicture}[x=3cm, y=1.2cm, 
	node/.style={minimum size=2em, fill=white, inner sep=2pt},
	impl/.style={draw,very thick,-latex},
	strict/.style={draw, -latex, double distance=2pt},
	nonimpl/.style={draw, very thick, dotted, -latex},
    equiv/.style={draw, very thick, latex-latex}
]

    \path[pattern=north east lines, pattern color=gray!60] 
        (-1.5, 0.4) 
        -- (-0.9, 0.4) 
        .. controls (-0.5, 0.4) and (-0.5, 2.5) .. (-0.2, 2.5) 
        -- (0.35, 2.5) 
        .. controls (0.6, 2.5) and (0.6, 4.6) .. (1.0, 4.6) 
        -- (2.5, 4.6) 
        -- (2.5, -0.5) -- (-1.5, -0.5) -- cycle;

    \draw[thick, dashed] 
        (-1.5, 0.4) 
        -- (-0.9, 0.4) 
        .. controls (-0.5, 0.4) and (-0.5, 2.5) .. (-0.2, 2.5) 
        -- (0.35, 2.5) 
        .. controls (0.6, 2.5) and (0.6, 4.6) .. (1.0, 4.6) 
        -- (2.5, 4.6);
	
	\node (AST) at (0, 0) {$\mathsf{AST} \equiv_c \mathsf{NON}$};
	\node (OPT) at (0, 1) {$\mathsf{OPT} \equiv_c \mathsf{HYP}$};
    \node (RCOLOR2) at (-1, 1) {$\mathsf{RCOLOR}_2$};
    \node (DNR) at (1, 1) {$\mathsf{DNC}$};
    \node (FIP) at (0, 2) {$\mathsf{FIP} \equiv_c 1\mbox{-}\mathsf{GEN}$};
    \node (AMT) at (0, 3) {$\mathsf{AMT} \equiv_c \Pi^0_1\mathsf{G}$};
	\node (COH) at (-1, 3) {$\mathsf{COH} \equiv_c \mathsf{CADS}$};
    \node (POS) at (-1, 2) {$\mathsf{POS}$};
    \node (SADS) at (0, 4) {$\mathsf{SADS}$};
    \node (WWKL) at (2, 2) {$\mathsf{WWKL} \equiv_c 1\mbox{-}\mathsf{RAN}$};
    \node (WWKL2) at (1.5, 4) {$2\mbox{-}\mathsf{WWKL} \equiv_c 2\mbox{-}\mathsf{RAN}$};
    \node (DNR2) at (1.2, 2.5) {$\mathsf{RRT}^2_2 \equiv_c 2\mbox{-}\mathsf{DNC}$};

    \draw[strict] (OPT) -- (AST);
    \draw[strict] (RCOLOR2) -- (AST);
    \draw[strict] (DNR) -- (AST);
    \draw[strict] (POS) -- (OPT);
    \draw[strict] (COH) -- (OPT);
    \draw[strict] (FIP) -- (OPT);
    \draw[strict] (AMT) -- (FIP);
    \draw[strict] (SADS) -- (AMT);
    \draw[strict] (WWKL) -- (DNR);
    \draw[strict] (WWKL2) -- (WWKL);
    \draw[strict] (WWKL2) -- (DNR2);
    \draw[strict] (DNR2) -- (DNR);
    \draw[strict] (DNR2) -- (OPT);
    \draw[strict] (WWKL2) -- (FIP);

\end{tikzpicture}

\caption{\label{fig:summary-ran} Classification based on randomness.}
\end{center}

\end{figure}

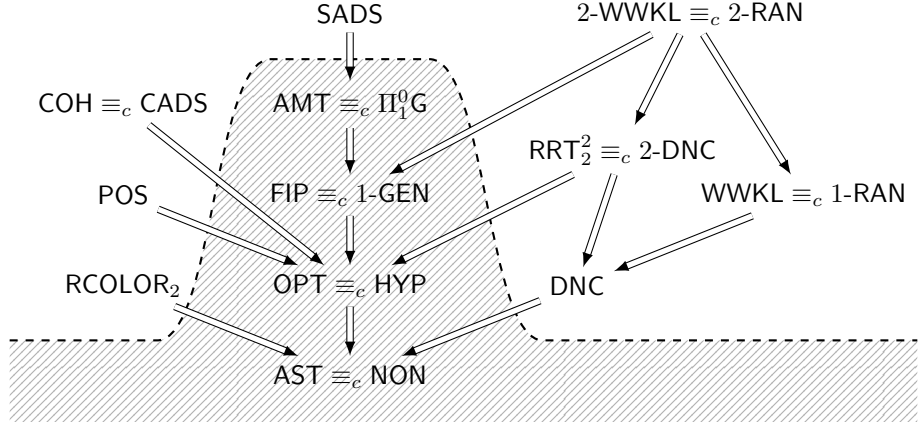
\begin{figure}[htbp]
\begin{center}
\begin{tikzpicture}[x=3cm, y=1.2cm, 
	node/.style={minimum size=2em, fill=white, inner sep=2pt},
	impl/.style={draw,very thick,-latex},
	strict/.style={draw, -latex, double distance=2pt},
	nonimpl/.style={draw, very thick, dotted, -latex},
    equiv/.style={draw, very thick, latex-latex}
]

    \path[pattern=north east lines, pattern color=gray!60] 
        (-1.5, 0.4) 
        -- (-0.85, 0.4) 
        .. controls (-0.6, 0.4) and (-0.6, 3.5) .. (-0.35, 3.5) 
        -- (0.35, 3.5) 
        .. controls (0.6, 3.5) and (0.6, 0.4) .. (0.85, 0.4) 
        -- (2.5, 0.4) 
        -- (2.5, -0.5) -- (-1.5, -0.5) -- cycle;

    \draw[thick, dashed] 
        (-1.5, 0.4) 
        -- (-0.85, 0.4) 
        .. controls (-0.6, 0.4) and (-0.6, 3.5) .. (-0.35, 3.5) 
        -- (0.35, 3.5) 
        .. controls (0.6, 3.5) and (0.6, 0.4) .. (0.85, 0.4) 
        -- (2.5, 0.4);
	
	\node (AST) at (0, 0) {$\mathsf{AST} \equiv_c \mathsf{NON}$};
	\node (OPT) at (0, 1) {$\mathsf{OPT} \equiv_c \mathsf{HYP}$};
    \node (RCOLOR2) at (-1, 1) {$\mathsf{RCOLOR}_2$};
    \node (DNR) at (1, 1) {$\mathsf{DNC}$};
    \node (FIP) at (0, 2) {$\mathsf{FIP} \equiv_c 1\mbox{-}\mathsf{GEN}$};
    \node (AMT) at (0, 3) {$\mathsf{AMT} \equiv_c \Pi^0_1\mathsf{G}$};
	\node (COH) at (-1, 3) {$\mathsf{COH} \equiv_c \mathsf{CADS}$};
    \node (POS) at (-1, 2) {$\mathsf{POS}$};
    \node (SADS) at (0, 4) {$\mathsf{SADS}$};
    \node (WWKL) at (2, 2) {$\mathsf{WWKL} \equiv_c 1\mbox{-}\mathsf{RAN}$};
    \node (WWKL2) at (1.5, 4) {$2\mbox{-}\mathsf{WWKL} \equiv_c 2\mbox{-}\mathsf{RAN}$};
    \node (DNR2) at (1.2, 2.5) {$\mathsf{RRT}^2_2 \equiv_c 2\mbox{-}\mathsf{DNC}$};

    \draw[strict] (OPT) -- (AST);
    \draw[strict] (RCOLOR2) -- (AST);
    \draw[strict] (DNR) -- (AST);
    \draw[strict] (POS) -- (OPT);
    \draw[strict] (COH) -- (OPT);
    \draw[strict] (FIP) -- (OPT);
    \draw[strict] (AMT) -- (FIP);
     \draw[strict] (SADS) -- (AMT);
    \draw[strict] (WWKL) -- (DNR);
    \draw[strict] (WWKL2) -- (WWKL);
    \draw[strict] (WWKL2) -- (DNR2);
    \draw[strict] (DNR2) -- (DNR);
    \draw[strict] (DNR2) -- (OPT);
    \draw[strict] (WWKL2) -- (FIP);

\end{tikzpicture}

\caption{\label{fig:summary-gen} Classification based on genericity.}
\end{center}

\end{figure}

\section*{Use of AI}

The authors used the large language model Claude Opus 5, by Anthropic, to translate the reverse-mathematical zoo by Astor and Dzhafarov~\cite{rmzoo} into a zoo under computable reducibility~\cite{zoodataset}. Besides the dataset, the article did not use AI.

\bibliographystyle{plain}
\bibliography{references}

\end{document}